\documentclass{amsart}
\usepackage{amssymb}
\usepackage{latexsym}
\usepackage{graphicx}
\usepackage{subfigure}
\usepackage{tikz}
\usepackage[linktocpage=true]{hyperref}

\hypersetup{ colorlinks   = true, 
urlcolor  = blue, 
linkcolor    = black, 
citecolor   = black 
}

\def\I{{\mathrm I}}
\def\II{{\mathrm{II}}}
\def\III{{\mathrm{III}}}

\def\L{{\Lambda}}
\def\t{{\theta}}
\def\l{{\lambda}}

\def\b{{\beta}}
\def\a{{\alpha}}
\def\e{{\varepsilon}}
\def\beq{\begin{equation}}
\def\eeq{\end{equation}}

\newcommand{\Z}{{\mathbb Z}}
\newcommand{\R}{{\mathbb R}}

\newcommand{\C}{{\mathbb C}}
\newcommand{\D}{{\mathbb D}}
\newcommand{\T}{{\mathbb T}}

\newcommand{\N}{{\mathbb N}}
\newcommand{\PP}{{\mathbb P}}

\newcommand{\CH}{{\mathcal H}}
\newcommand{\CI}{{\mathcal I}}

\newcommand{\CN}{{\mathcal N}}

\newcommand{\CZ}{{\mathcal Z}}
\newcommand{\CU}{{\mathcal U}}

\newtheorem{theorem}{Theorem}
\newtheorem{remark}{Remark}
\newtheorem{lemma}{Lemma}
\newtheorem{defi}{Definition}
\newtheorem{prop}{Proposition}

\newtheorem{corollary}{Corollary}
\newtheorem*{thma}{Theorem A}
\newtheorem*{thmb}{Theorem B}
\begin{document}

\title[Cantor spectrum for smooth potentials]{Cantor spectrum for a class of $C^2$ Quasiperiodic Schr\"odinger Operators}

\author[Y.\ Wang]{Yiqian Wang}

\address{Department of Mathematics, Nanjing University, Nanjing 210093, China}

\email{yiqianw@nju.edu.cn}

\thanks{Y.W. was supported by the National Natural Science Foundation of China grant No. 11271183.}

\author[Z.\ Zhang]{Zhenghe Zhang}

\address{Department of Mathematics, Rice University, Houston, TX~77005, USA}

\email{zzhang@rice.edu}

\thanks{This work was partially conducted during the period Z.Z. was supported by NSF grant DMS-1316534}

\begin{abstract} We show that for a class of $C^2$ quasiperiodic
potentials and for any \emph{Diophantine} frequency, the spectrum is Cantor. Our approach is of purely dynamical systems, which depends on a detailed analysis of asymptotic stable and unstable directions. We also apply it to general $\mathrm{SL}(2,\R)$ cocycles, and obtain that \emph{uniform hyperbolic} systems form a open and dense set in some one-parameter family.
\end{abstract}

\maketitle

\setcounter{tocdepth}{1}
\tableofcontents

\section{Introduction}\label{s.introduction}

Consider the family of Schr\"odinger operators $H_{\alpha,\l,v,x}$ on $\ell^2(\Z)\ni u=(u_n)_{n\in\Z}$:
\beq\label{operator}
(H_{\alpha,\lambda,v,x}u)_n=u_{n+1}+u_{n-1}+\lambda v(x+n\alpha)u_n.
\eeq
Here $v\in C^r(\R/\Z,\R),r\in\N\cup\{\infty,\omega\}$ is the potential, $\lambda\in\R$ coupling constant, $x\in\R/\Z$ phase, and $\alpha\in\R/\Z$ frequency. For simplicity, we may sometimes left $\a,\l, v$ in $H_{\alpha,\lambda,v,x}$ implicit.

Due to a theorem of Johnson \cite{johnson}, it's well-known that for irrational $\a$, the spectrum of the family of operators $H_{x}$ is phase-independent. So we may let $\Sigma_{\a,\l,v}$ denote the common spectrum in this case. Then it is well-known that
\beq\label{interval-contains-spectrum}
\Sigma_{\a,\l,v}\subset [-2+\l\inf v, 2+\l\sup v].
\eeq

Consider the eigenvalue equation $H_{x}u=Eu.$ Then there is an associated cocycle map which is denoted as
$A^{(E-\lambda v)}\in C^{r}(\mathbb R/\mathbb Z, \mathrm{SL}(2,\mathbb R))$, and is given by
\beq\label{schrodinger-cocycle-map}
A^{(E-\lambda v)}(x)=\begin{pmatrix}E-\lambda v(x)& -1\\1& 0\end{pmatrix}.
\eeq
Then $(\alpha, A^{(E-\lambda v)})$ defines a family of dynamical systems on $(\R/\Z)\times\R^2$, which is given by
\beq\label{schrodinger-cocycle}
(x,w)\mapsto (x+\alpha, A^{(E-\l v)}(x)w),
\eeq
and is called the Schr\"odinger cocycle. The $n$th iteration of dynamics is denoted by $(\alpha, A^{(E-\l v)})^n=(n\alpha, A^{(E-\l v)}_n)$. Thus,
$$
A^{(E-\l v)}_n(x)=\begin{cases}A^{(E-\l v)}(x+(n-1)\a)\cdots A^{(E-\l v)}(x), &n\ge 1;\\ Id, &n=0;\\ [A^{(E-\l v)}_{-n}(x+n\a)]^{-1}, &n\le -1.\end{cases}
$$

The relation between operator and cocycle is the following. $u\in\C^{\Z}$ is a solution of the equation $H_{\lambda,x}u=Eu$ if and only if
$$
A^{(E-\l v)}_n(x)\binom{u_{0}}{u_{-1}}=\binom{u_{n}}{u_{n-1}},\ n\in\Z.
$$
This says that $A^{(E-\l v)}_n$ generates the $n$-step transfer matrices for the operator (\ref{operator}).

\vskip 0.3cm
\subsection{Statement of main result and some review}\label{ss.main-results}
In this paper, from now on, we assume $v\in C^2(\mathbb R/\mathbb Z,\mathbb R)$ satisfy the following conditions:
\begin{itemize}
\item $\frac{dv}{dx}=0$ at exactly two points, one is minimal and the other maximal, which are denoted by $z_1$ and $z_2$.
\item these two extremals are non-degenerate. In other words, $\frac{d^2v}{dx^2}(z_j)\neq0$ for $j=1,2$.
\end{itemize}

Fix two positive constants $\tau, \gamma$. We say $\a$ satisfying a \emph{Diophantine} condition $DC_{\tau,\gamma}$ if
$$
|\a-\frac{p}{q}|\ge \frac{\gamma}{|q|^{\tau}} \mbox{ for all } p, q\in \Z \mbox{ with } q\neq 0.
$$
It is a standard result that for any $\tau>2$,
$$
DC_\tau:=\bigcup_{\gamma>0}DC_{\tau,\gamma}
$$
is of full Lebesgue measure. From now on, we fix an arbitrary $\tau>2$ and consider any fixed choice of $\a\in DC_\tau$. Then, we would like to show the following result.
\begin{thma}[Main Theorem]\label{t.main}
  Let $\a$ and $v$ be as above. Consider the Schr\"odinger operators with potential $v$ and coupling constant $\l$. Then there exists a $\l_0=\l_0(\a,v)>0$ such that
\begin{center}
$\Sigma_{\a,\l,v}$ is a Cantor set
\end{center}
  for all $\l>\l_0$.
\end{thma}

The geometric structure of spectrum has been one of the central topics for the Schr\"odinger operators (\ref{operator}) for a very long time. For rational $\alpha=\frac pq$, it is well known that spectrum consists of $q$ non-degenerate intervals which are called spectral bands. For irrational $\alpha$, the spectrum is widely expected to be Cantor. We review some of the results in the following.

The most intensively studied model is probably the Almost Mathieu Operator, where $v(x)=\cos2\pi x$. Then the famous `The Ten Martini Problem' conjectured that for this operator, the spectrum must be Cantor for all non-zero couplings and for all irrational frequency. This conjecture has been completely settled by Avila-Jitomirskaya \cite{AJ}. We refer the authors to \cite{AJ} for further references regarding its long history and numerous earlier partial results.

For general real analytic potentials, Eliasson \cite{eliasson} proved that for a fixed Diophantine frequency and for small couplings, the spectrum is Cantor for generic potentials under suitable analytic topology. On the other hand, Goldstein-Schlag \cite{goldsteinschlag} proved that for any real analytic potential, in the region of positive Lyapunov exponents, spectrum is Cantor for almost every frequency.

For $C^0$ potentials, Avila-Bochi-Damanik \cite{avilabochidamanik} proved that for any irrational frequency, and for $C^0$ generic potentials, the spectrum is a Cantor set.

If one lower the regularity even more, then Damanik-Lenz \cite{DL1, DL2} obtained for a large class of step functions, the spectrum is of zero Lebesgue measure which clearly implies Cantor Spectrum. These type of potentials are in particular related to Sturmian subshift, or even more special, the Fibonacci subshift, which has also been intensively studied. See \cite{DL1, DL2} for further references and more detailed description.

For non-Cantor examples, Avila-Damanik-Zhang \cite{aviladamanikzhang} showed that for a dense subset in the joint space of frequency and $C^0$ potentials, the corresponding spectrum contain non-degenerate interval while frequencies are irrational. Those are the first examples of this kind for non-periodic quasiperiodic potentials.

For $C^r$ potentials with $1\le r\le \infty$, the model in Theorem~A was in fact first studied by Sinai~\cite{sinai}. To the best of our knowledge, this model so far is the only $C^r$, $1\le r\le \infty$, quasiperiodic potentials that one is able to show Cantor spectrum. The smooth case is considered to be difficult. Because in this case, one cannot get around the small divisor type of problems by tools or techniques like subharmonicity (in $C^\omega$ case) or Kotani Theory (in the $C^0$ or step functions cases). Arithmetic properties, like \emph{Diophantine} or \emph{Brjuno} conditions, and complicated induction are always necessary for whatever type of problems in this case.

We would like to emphasize that though the flavor of our approach is of purely dynamical systems, we are benefited from \cite{sinai} for the idea of `resonance leads to spectral gaps'. This idea is also used in \cite{goldsteinschlag}.

\subsection{Idea of proof of the Main Theorem}\label{ss.idea-of-proof}

We first define \emph{uniformly hyperbolic} systems ($\CU\CH$).

Consider for any $\beta\in\R/\Z$ and any cocycle map $B\in C^0(\R/\Z,\mathrm{SL}(2,\R))$ the dynamical system $(\beta,B):\R/\Z\times\R^2\rightarrow\R/\Z\times\R^2$ as we defined in (\ref{schrodinger-cocycle}). For any $M\in SL(2,\R)$, let $\bar M:\R/(\pi\Z)=\R\PP^1\rightarrow\R\PP^1$ be the induced map on the real projective space. Let $\bar B\in C^0(\R/\Z,\mathrm{PSL}(2,\R))$ be induced map such that $\bar B(x)=\overline{B(x)}$. Then we have the following definition.

\begin{defi}
We say $(\beta,B)$ is \emph{uniformly hyperbolic} ($\CU\CH$) if there exists two funcitons $\bar s, \bar u:\R/\Z\rightarrow\R\PP^1$ such that the following holds:
\begin{itemize}
\item $\bar B(x)\cdot \bar s(x)=\bar s(x+\beta)$, $\bar B(x)\cdot \bar u(x)=\bar u(x+\beta)$.
\item There exist $c>0$ and $\rho>1$ such that for any unit vector $w^s\in\bar s(x)$ and $w^u\in\bar u(x)$, it holds that
 $$
 \|B_n(x)w^s\|,\ \|B_{-n}(x)w^u\|<c\rho^{-n}
 $$
 for all $n\ge 1$ and for all $x\in \R/\Z$. Here we also consider $\bar s(x)$ and $\bar u(x)$ as one dimensional subspace of $\R^2$.
\end{itemize}
\end{defi}

Here $\bar s$ is the so-call stable direction and $\bar u$ the unstable direction. Note $\CU\CH$ is an open condition in $C^0$ topology. Throughout this paper, $(\b,B)\in\CU\CH$ will be equivalent to the statement `$(\b,B)$ is \emph{uniformly hyperbolic}', and $(\b,B)\notin\CU
\CH$ to `$(\t,B)$ is not \emph{uniformly hyperbolic}'. Then the following basic relation between spectral theory of Schr\"odinger operators and the dynamics of Schr\"odinger cocycles is due to Johnson \cite{johnson}.

\begin{theorem}\label{t.basic}
For irrational $\a$, it holds that
\beq\label{t.johnson}
\Sigma_{\a,\l,v}=\{E:\ (\a,A^{(E-\l v)})\notin\CU\CH\}.
\eeq
\end{theorem}
Clearly, Theorem~\ref{t.basic} reduce the proof of Theorem~A to the proof of the following theorem.
\begin{thmb}\label{t.main1}
Let $\a$, $v$ and $\l$ be as in Theorem~A. Let $S\subset\R$ be any nondegenerate interval. Then there exists some $E\in S$ such that $(\a, A^{(E-\l v)})\in \CU\CH$.
\end{thmb}

The following equivalent condition for $\CU\CH$ is also well-known, see, e.g. \cite{yoccoz}.
\begin{prop}\label{p.ueg}
Let the dynamical system $(\beta, B)$ be as above. Then $(\b, B)\in\CU\CH$ if and only if the following \emph{Uniform Exponential Growth} condition holds. There exist $c>0$ and $\rho>1$ such that
\beq\label{ueg}
\|B_n(x)\|>c\rho^{|n|}
\eeq
for all $n\in\Z$ and for all $x\in\R/\Z$.
\end{prop}
See also \cite{zhang2} for more recent proofs of the phase-independence of the spectrum, Theorem~\ref{t.johnson}, and Proposition~\ref{p.ueg}. In particular, the proof of the Proposition~\ref{p.ueg} in \cite{zhang2} is also close in spirit of the proof of the Main Theorem in this paper.

Basically, we can define a pair of asymptotic stable and unstable directions, $s_n$ and $u_n$, which are the most contraction directions of $A^{(E-\l v)}_n(x)$ and $A^{(E-\l v)}_{-n}(x)$. It's not difficult to see that $(\a,A^{(E-\l v)})\in\CU\CH$ if and only if $|s_n(x)-u_n(x)|$ is bounded away from zero with a distance of order $\|A^{(E-\l v)}_n(x)\|^{-1}$ for some $n\ge 1$ (See, e.g. \cite[Theorem 1]{zhang2}). The reason we call them asymptotic stable and unstable directions is just that in case of positive Lyapunov exponents, they converge to the real stable and unstable directions as $n$ goes to infinity. In particular, the convergence is of uniform fashion in case of $\CU\CH$.

It is clear that for each $E$, the dangerous points are exactly those $x$ such that
$$
s_n(x,E)-u_n(x,E)=0.
$$
These points are called $n$-step `\emph{critical points}'. Since we necessarily starts with some $E$ that the $1$-step \emph{critical points} exist, one need some mechanism such that \emph{critical points} disappear at some step $n$. It turns out `\emph{strong resonance} between orbits of different $n$-step critical points' is one of such mechanisms. Here \emph{strong resonance} means the orbit of one $n$-step critical point getting sufficiently close to another $n$-step critical points within some time that is not too large.

Then, roughly speaking, what we do is to develop some induction scheme with following properties. We find a times sequence $\{r_n\}_{n\in\N}\subset\Z_+$, $\lim_{n\rightarrow\infty}r_n=\infty$, such that:

\begin{itemize}
\item for any interval $J_0$ of energies $E$, as the induction moving forward, \emph{strong resonance} occurs between orbits of different zeros of
$$
s_{r_n}(\cdot,E)-u_{r_n}(\cdot,E)
$$
at some step $n$ for some $E\in J_0$;
\item strong resonance leads to the separation of asymptotic stable and unstable directions with distance of order $\|A^{(E-\l v)}_{r_n}(x)\|^{-1},\ x\in I_n$.
\end{itemize}

It turns out that the time sequence $r_n$ is nothing other than some `\emph{return times}' of some sequence of intervals, $I_n(E)$, of base space with $\lim_{n\rightarrow\infty}|I_n(E)|=0$. And to obtain two properties above, we need to control the geometric properties of $s_n$ and $u_n$ both as functions of $x$ and of $E$. The main difficulty comes from `\emph{resonance}', the occurrence of which leads to drastic change of  $s_n(x,E)$ and $u_n(x,E)$.

To deal with this issue, we need to define some suitable class of \emph{`nice'} $C^2$ functions so that for each $E$,
$$
s_{r_n}(\cdot,E)-u_{r_n}(\cdot,E):I_n(E)\rightarrow\R\PP^1
$$
falls into this class at each induction step, see Definition \ref{d.type-of-functions} in Section 2 for details.

The idea of locating the $n$-step \emph{critical points} and of estimating the geometric properties of $s_n$ and $u_n$ go back to Young \cite{young}, which is close in spirit to the Benedicks-Carleson \cite{benedickscarleson} type of techniques for Hen\'on maps, and has been successfully applied to Schr\"odinger cocycle by Zhang \cite{zhang}, Wang-You \cite{youwang1, youwang2}, and Wang-Zhang \cite{Wang-Zhang}.

In fact, the proof of Theorem~B in this paper is based on the induction scheme developed in \cite{Wang-Zhang}, where the authors have already obtained uniform positivity and continuity of the Lyapunov exponents for the same model as in Theorem~A.

For the energy parameter $E$, in this paper, we need to get the lower bound of
$$
\left|\frac{\partial(s_{r_n}-u_{r_n})}{\partial E}(x,E)\right|
$$
in each induction step so that the distance between zeros of $s_n(\cdot,E)-u_n(\cdot,E)$ varies with large enough rate in $E$. It turns out that, roughly speaking, one needs the uniform exponential growth of
$$
\|A^{(E-\l v)}_{\pm r_n}(x)\|,\ x\in I_n(E),\ E\in J_0
$$
to get all the necessary estimates, which is essentially nothing other than the uniform positivity of the corresponding Lyapunov exponents. Fortunately, almost all the main estimates mentioned above have been done in \cite{Wang-Zhang}.

\subsection{Generalization and further comments}\label{ss.generalization}

The main advantage of our approach is that it's of purely dynamical systems, which is in particular not restricted to the Schr\"odinger cocycle category. For instance, we may consider a function $\psi\in C^2(\R/\Z, \R)$ satisfying all the properties of $v$ in Theorem~A. In addition, we also assume that $\sup\psi-\inf\psi<\pi$. Then consider the one-parameter family of cocycle map $B^{(\t,\l)}\in C^2(\R/\Z,\mathrm{SL}(2,\R))$ such that
$$
B^{(\t,\l)}(x)=\begin{pmatrix}\l &0\\0&\l^{-1}\end{pmatrix}\cdot R_{\psi(x)}\cdot R_\t,\ \t\in\R,
$$
where
$$
R_\phi=\begin{pmatrix}\cos\phi &-\sin\phi\\ \sin\phi &\cos\phi\end{pmatrix}\in\mathrm{SO}(2,\R)
$$
is the rotation matrix with rotating angle $\phi$. Then we have the following corollary of the proof of Theorem~B. In fact, combining with \cite[Corollary 2,5]{Wang-Zhang}, we have the following results. Let $L(\a,\t,\l)$ be the Lyapunov exponents of the dynamical systems $(\a,B^{(\t,\l)})$. In other words,
$$
L(\a,\t,\l)=\lim_{n\rightarrow\infty}\frac1n\int_{\R/\Z}\log\|B^{(\t,\l)}_n(x)\|dx.
$$
Then we have the following results which give relatively complete description of dynamical behavior of the one-parameter family (over $\t$) of dynamical systems $(\a,B^{(\t,\l)})$ for large $\l$, hence, illustrate the power of our techniques.
\begin{corollary}\label{c.rotation_family}
Fix a $\a\in DC_\tau$. Let $\psi\in C^2(\R/\Z,\R)$ and $B^{(\t,\l)}\in C^2(\R/\Z,\mathrm{SL}(2,\R))$ be as above. Then there exists a $\l_0=\l_0(\a,\psi)$ such that for all $\l>\l_0$:
\begin{itemize}
\item $\left\{\t\in\R:(\a,B^{(\t,\l)})\in\CU\CH\right\}$ is open and dense.
\item $L(\a,\t,\l)>\frac{99}{100}\log \l$ for all $\t\in\R$.
\item There exists a $\sigma\in(0,1)$ depending on $\a$ such that
      $$
      |L(\a,\t,\l)-L(\a,\t',\l)|<Ce^{-c(\log|\t-\t'|^{-1})^\sigma}.
      $$
\end{itemize}
Moreover, we also have
$$
\lim_{\l\rightarrow\infty}\mathrm{Leb}\{\t\in[0,\pi):(\a,B^{(\t,\l)})\notin\CU\CH\}=\mathrm{Leb}[\psi(\R/\Z)].
$$

\end{corollary}
Corollary~\ref{c.rotation_family} can in particular be applied to get the density of $\CU\CH$ in the one-parameter family of some $C^2$ Szeg\H o cocycles, which arise naturally in the study of orthogonal polynomials on the unit circle. See appendix Section~\ref{s.generalization}.

We also have the following generalization which shows that the geometric structure of spectrum is a local property with respect to the value of $v$.
\begin{corollary}\label{c.local}
  Fix a Diophantine number $\a$ and consider a potential $v(x)\in C^2(\R/\Z,\R)$. Assume the interval $[E_1, E_2]$ satisfies that for each $E$ in it, the equation $v(x)=E$ has at most two solutions. Moreover, we assume that $\frac{d^2v}{dx^2}(x_0)\neq0$ if $x_0$ is the only solution. Then there exists a $\l_0=\l_0(a,v)$ such that for all $\l>\l_0$
   $$
   \{E: (\a, A^{(E-\l v)})\in\CU\CH\}
   $$
   is open and dense in $[E_1,E_2]$.
\end{corollary}
See Appendix Section~\ref{s.generalization} for more detailed discussion regarding the Corollaries above.\\

Another advantage of our approach is that, we are able to fix both frequency and potential, which is dynamically more natural. For example, for those works mentioned in Section~\ref{ss.main-results}, the generic results of \cite{eliasson,avilabochidamanik} necessary involves variation of potentials, \cite{goldsteinschlag} need to vary frequency, and the non-Cantor results of \cite{aviladamanikzhang} varies both frequency and potential.


Finally, it's possible for our approach to do the following further developments. We may get some estimates on the size of spectral gaps. It may also be useful in investigating the dry version of Cantor spectrum for smooth potentials. In otherwise words, we may investigate for which type of smooth potentials, the spectral gaps corresponding to the fibred ration numbers $\{\frac{k\a}{2}\in\R/\Z\}_{k\in Z\setminus\{0\}}$ are all opened up. Moreover, as commented in \cite{Wang-Zhang}, the idea of study the asymptotic stable and unstable directions is not restricted to the models in Theorem~A. For example, we may consider more general smooth potentials, relax the Diophantine condition to \emph{Brjuno} or even \emph{weak Liouville} conditions, or to some other type of base dynamics such as doubling map on the unit circle.

\subsection{Structure of the paper}

For the remaining part of this paper, Preliminary Section~\ref{s.preliminary} will mainly be some detailed description of consequences from the Induction Theorem of \cite{Wang-Zhang}. These consequences are the cornerstones of this paper. In particular, as mentioned in Section~\ref{ss.idea-of-proof}, we will define a class of nice $C^2$ functions, and state the theorem that for all energies $E$, the function $s_{r_n}(\cdot,E)-u_{r_n}(\cdot,E):I_n\rightarrow\R\PP^1$ falls into one of them. In particular, we will explain in detail how the consequence of the occurrence of \emph{resonance}. In Section~\ref{s.proof-of-main-theorem}, we will prove Theorem~B, hence the Main Theorem A. In appendix Section~\ref{s.generalization}, we will discuss some generalization.

\section{Preliminary and review of \cite{Wang-Zhang}}\label{s.preliminary}
This section is basically a review of some of the results of \cite{Wang-Zhang} that will be explicitly used in this paper. More detailed description and proof of these results, in particular the proof of Lemma~\ref{l.reduce-form}--\ref{l.s-deri-resonance} and Theorem~\ref{t.iteration}, can be found in \cite{Wang-Zhang}.

From now on, if not stated otherwise, let $C,\ c$ be some universal positive constants depending only on $v$ and $\a$, where $C$ is large and $c$ small. For two constants $a,b>0$, by $a \gg b $ or $b\ll a$, we mean that $a$ is sufficiently larger than $b$.

Recall that for $\theta\in \R/(2\pi \Z)$, we have the following rotation matrix
$$
R_\t=\begin{pmatrix}\cos\t&-\sin\t\\ \sin\t&\cos\t\end{pmatrix}\in\mathrm{SO}(2,\mathbb R).
$$

Define the map
$$
s:\mathrm{SL}(2,\R)\rightarrow\R\PP^1=\R/(\pi\Z)
$$
so that $s(A)$ is the most contraction direction of $A\in\mathrm{SL}(2,\R)$. Let $\hat s(A)\in s(A)$ be an unit vector. Thus, $\|A\cdot\hat s(A)\|=\|A\|^{-1}$. Abusing the notation a little, let
$$
u:\mathrm{SL}(2,\R)\rightarrow\R\PP^1=\R/(\pi\Z)
$$
be that $u(A)=s(A^{-1})$. Then for $A\in\mathrm{SL}(2,\R)$, it is clear that
\beq\label{polar-form-single-matrix}
A=R_u\cdot\begin{pmatrix}\|A\| &0\\ 0 &\|A\|^{-1}\end{pmatrix}\cdot R_{\frac\pi2-s},
\eeq
where $s,u\in [0,2\pi)$ are some suitable choices of angles correspond to the directions $s(A),u(A)\in \R/(\pi\Z)$.

Set $t=\frac E\l$. Then instead of proving Theorem~B directly, we will use the following form of the cocycle map.
\begin{lemma}\label{l.reduce-form}
 Let $\CI\subset\R$ be some compact interval. For $x\in\mathbb R/\mathbb Z$ and $t\in\CI$, define  the following cocycles map
\beq\label{polar-decom}
A(x,t)=\Lambda(x)\cdot R_{\phi(x,t)}:=\begin{pmatrix}\l(x)&0\\0&\l^{-1}(x)\end{pmatrix}\cdot\begin{pmatrix}\frac{t-v(x)}{\sqrt{(t-v(x))^2+1}}&\frac{-1}{\sqrt{(t-v(x))^2+1}}\\
\frac{1}{\sqrt{(t-v(x))^2+1}}&\frac{t-v(x)}{\sqrt{(t-v(x))^2+1}}\end{pmatrix},
\eeq
where $\cot\phi(x,t)=t-v(x)$. Assume
\beq\label{norm1-deri-control}
\l(x)>\l,\ \left|\frac{d^m\l(x)}{dx^m}\right|<C\l,\ m=1,2.
\eeq
Then to prove Theorem~B, it is enough to prove the corresponding results for (\ref{polar-decom}).
\end{lemma}

From now on, let $A=A(x,t)$ be as in (\ref{polar-decom}). Abusing the notation a little bit, for $n\ge 1$, we define
$$
s_n(x,t)=s[A_n(x,t)],\ u_n(x,t)=s[A_{-n}(x,t)].
$$
We call $s_n$ (respectively, $u_n$) the $n$-step \emph{stable} (respectively, \emph{unstable}) direction. It is not very difficult to see that they converge to the stable and unstable directions in case one has a positive Lyapunov exponent, see, for example, the proof of \cite[Theorem 1]{zhang2}.

Obviously, we have that $u_1(x,t)=0$ and
$$
s_1(x,t)=\frac\pi2-\phi(x,t)=\frac\pi2-\cot^{-1}[t-v(x)]=\tan^{-1}[t-v(x)] .
$$
We define $g_1(x,t)=s_1(x,t)-u_1(x,t)$. Thus, it clearly holds that
\beq\label{g1}
g_1(x,t)=\tan^{-1}[t-v(x)].
\eeq

We may restrict the $t$ to the following interval:
$$
t\in J:=[\inf v-\frac 2\l_0, \sup v+\frac 2\l_0] \mbox{ for all }\l>\l_0.
$$

We first define the following three types of functions, which basically classify for each fixed $t$, all the possibilities of $g_n(\cdot,t)$ as function of $x$.

Let $B(x,r)\subset\T$ be the ball centered around $x\in\T$ with radius $r$. For a connected interval $J\subset\T$ and constant $0<a\le1$, let $aJ$ be the subinterval of $S$ with the same center and whose length is $a|S|$. Let $I=B(0,r)$ and $l_0$ satisfy $l_0\gg r^{-1}\gg 1$. Consider for some small $\beta>0$, a function $l: I\rightarrow \R$ such that
\beq\label{large-parameter}
l(x)>l_0\mbox{ and }\frac{dl^m}{dx^m}(x)<l(x)^{1+\beta},\ \forall x\in I,\ m=1,2.
\eeq
Then we define the following types of functions, see Figure~\ref{f.graph-type123} for their graphs.
\begin{defi}\label{d.type-of-functions}
Let $I$ and $l$ be as above. Let $f\in C^2(I,\mathbb R\mathbb P^1)$. Then
\vskip0.2cm
\begin{itemize}
\item \noindent {\bf $f$ is of type $\mathrm{I}$} if we have the following.
    \begin{itemize}
    \item $\|f\|_{C^2}<C$ and $f(x)=0$ has only one solution, say $x_0$, which is contained in $\frac{I}{3}$;
    \item $\frac{df}{dx}=0$ has at most one solution on $I$ while $|\frac{df}{dx}|>r^2$ for all $x\in B(x_0,\frac r2)$;
    \item let $J\subset I$ be the subinterval such that $\frac{df}{dx}(J)\cdot\frac{df}{dx}(x_0)\le 0$, then $|f(x)|>cr^3$ for all $x\in J.$
    \end{itemize}
     Let $\mathrm{I}_{+}$ denotes the case $\frac{df}{dx}(x_0)> 0$ and $\mathrm{I}_{-}$ for $\frac{df}{dx}(x_0)<0$.

\vskip0.2cm
\item \noindent {\bf $f$ is of type $\mathrm{II}$} if we have the following.
     \begin{itemize}
     \item $\|f\|_{C^2}<C$ and $f(x)=0$ has at most two solutions which are in $\frac{I}{2}$;
     \item $\frac{df}{dx}(x)=0$ has one solution which is contained in $\frac{I}{2}$;
     \item $f(x)=0$ has one solution if and only if it is the $x$ such that $\frac{df}{dx}(x)=0$;
     \item $\left|\frac{d^2f}{dx^2}\right|>c$ whenever $|\frac{df}{dx}|<r^2$.
     \end{itemize}

\vskip0.2cm
\item \noindent {\bf $f$ is of type $\mathrm{III}$} if for $l: I\rightarrow\R$ as in (\ref{large-parameter}),
\beq\label{type3}
f=\tan^{-1}\left(l^2[\tan f_1(x)]\right)-\frac\pi2+f_2.
\eeq
Here either $f_1$ is of type $\mathrm{I}_+$ and $f_2$ of type $\mathrm{I}_-$, or $f_1$ is of type $\mathrm{I}_-$ and $f_2$ of type $\mathrm{I}_+$.
\end{itemize}
\end{defi}

\begin{figure}
\begin{center}
\begin{tikzpicture}[yscale=1.5]
\draw [->] (-1.5,0) -- (1.5,0);
\draw [->] (0,-0.5) -- (0,1.3);
\draw [thick,domain=-1.4:1.4] plot (\x, {0.1*\x*\x+0.3*\x});
\node [above] at (0,1.3) {$f(x)$};
\node [right] at (1.5,0) {$x$};
\node[align=left, below] at (0,-0.5)%
{Type $\I$};
\end{tikzpicture}\ \ \ \
\begin{tikzpicture}[yscale=1.5]
\draw [->] (-1.5,0) -- (1.5,0);
\draw [->] (0,-0.5) -- (0,1.3);
\draw [thick,domain=-1.4:1.4] plot (\x, {0.25*\x*\x-0.1});
\node [above] at (0,1.3) {$f(x)$};
\node [right] at (1.5,0) {$x$};
\node[align=left, below] at (0,-0.5)%
{Type $\II$};
\end{tikzpicture}\ \ \ \
\begin{tikzpicture}[yscale=0.7]
\draw [->] (-1.5,0.6) -- (1.5,0.6);
\draw [->] (0,0) -- (0,3.75);
\draw [thick,domain=-1:1] plot (\x, {1/180*pi*atan(20*tan(\x r))+0.5*pi-0.5*(\x-0.9)});
\draw [semithick,domain=-1.2:1.2] plot (\x, {0*\x+pi});
\node [above] at (0,3.75) {$f(x)$};
\node [right] at (1.5,0.6) {$x$};
\node [left] at (-1.2,3.14) {$\pi$};
\node[align=left, below] at (0,-0.1)%
{Type $\III$};
\end{tikzpicture}
\caption{Graph of Type $\I$, $\II$ and $\III$ functions}
\label{f.graph-type123}
\end{center}
\end{figure}
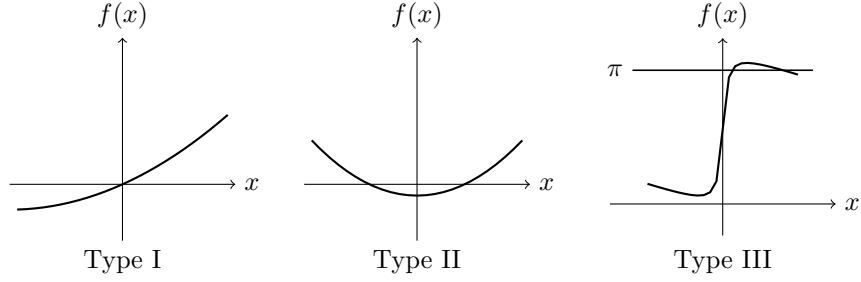

Now we investigate more properties of type $\III$ function. Without loss of generality, let $f$ be as in (\ref{type3}) with $f_1$ be type $\mathrm{I}_+$ and $f_2$ type $\mathrm{I}_-$ throughout this section. We may further assume that $f_1(0)=0$ and $f_2(d)=0$ with $0\le d\le \frac 23r$. Let
$$
X=\{x\in I: \mathbb R\mathbb P^1\ni |f(x)|=\inf_{y\in I}|f(y)|\}.
$$
Then it is easy to see that $X$ contains at most two points, say $X=\{x_1,x_2\}$ with $x_1\le x_2$. Then we have the following lemma.

\begin{lemma}\label{l.s-deri-resonance}
Let $f$ be of type $\mathrm{III}$. Let $r^{2}\le \eta_j\le r^{-2}$, $0\le j\le 4$. Then
\beq\label{type3-zero}
|x_1|<Cl^{-\frac 34},\ |x_2-d|<Cl^{-\frac 34}.
\eeq
In particular, if $f(x_1)=f(x_2)=0$, then
\beq\label{type3-zero-order}
0<x_1\le x_2<d.
\eeq
If $f(x_1)=f(x_2)\neq 0$, then
\beq\label{type3-zero-1}
x_1=x_2.
\eeq

There exist two distinct points $x_3,\ x_4\in B(x_1,\eta_0l^{-1})$ such that $\frac{df}{dx}(x_j)=0$ for $j=3,4$, and $x_3$ is
 a local minimum with \beq\label{type3-nonzero-minimum}
f(x_3)>\eta_1l^{-1}-\pi.
\eeq
 Moreover, we have the following. If $d\ge \frac r3$, then it holds that
\beq\label{type3-1}
|f(x)|>cr^3,\ x\notin B(x_1,Cl^{-\frac14})\cup B(x_2,\frac r4);\ \|f-f_2\|_{C^1}<Cl^{-\frac32},\ x\in B(x_2,\frac r4).
\eeq

If $d<\frac r3$, then it holds that
\beq\label{type3-nondegeneracy}
\left|\frac{d^2f}{dx^2}(x)\right|>c \mbox{ whenever } \left|\frac{df}{dx}(x)\right|\le r^2\mbox{ for } x\in B(X,\frac r6)
\eeq
and $|f(x)|>cr^3$ for all $x\notin B(X,\frac r6)$.\\

Finally, we have the following bifurcation as $d$ varies. There is a $d_0=\eta_2 l^{-1}$ such that:
\begin{itemize}
\item if $d>d_0$, then $f(x)=0$ has two solutions;
\item if $d=d_0$, then $f(x)=0$ has exactly one tangential solution. In other words, $x_1=x_2=x_4$ and $f(x_4)=0$;
\item if $0\le d<d_0$, then $f(x)\neq 0$ for all $x\in I$. Moreover, we have
      $$
      \min_{x\in I}|f(x)|=-\eta_3l^{-1}+\eta_4d.
      $$
\end{itemize}
See Figure~\ref{f.bifurcation} for the bifurcation procedure.
\end{lemma}

\begin{figure}
\begin{center}
\begin{tikzpicture}[yscale=0.7]
\draw [->] (-1.5,0.3) -- (1.5,0.3);
\draw [->] (0,0) -- (0,4);
\draw [thick,domain=-1:1] plot (\x, {1/180*pi*atan(20*tan(\x r))+0.5*pi-0.5*(\x-0.9)});
\draw [semithick,domain=-1.2:1.2] plot (\x, {0*\x+pi});
\draw [dotted, thick, red] (-0.31,0.3) -- (-0.31,0.8);
\draw [dotted, thick, red] (0.32,0.3) -- (0.32,3.3);
\draw[fill=red] (-0.31,0.3) circle [radius=0.055];
\node [below] at (-0.31,0.3) {$x_3$};
\draw[fill=blue] (0.14,pi) circle [radius=0.055];
\node [above] at (0.2,pi) {$x_1$};
\draw[fill=red] (0.32,0.3) circle [radius=0.055];
\node [below] at (0.32,0.3) {$x_4$};
\draw[fill=blue] (0.75,pi) circle [radius=0.055];
\node [above] at (0.75,pi)  {$x_2$};
\node [above] at (0,4) {$f(x)$};
\node [right] at (1.5,0.3) {$x$};
\node [left] at (-1.2,3.14) {$\pi$};
\node[align=left, below] at (0,-0.3)%
{$d>d_0$};
\end{tikzpicture}\ \ \ \
\begin{tikzpicture}[yscale=0.7]
\draw [->] (-1.5,0.3) -- (1.5,0.3);
\draw [->] (0,0) -- (0,4);
\draw [thick,domain=-1:1] plot (\x, {1/180*pi*atan(20*tan(\x r))+0.5*pi-0.5*(\x-0.57)});
\draw [semithick,domain=-1.2:1.2] plot (\x, {0*\x+pi});
\node [above] at (0,4) {$f(x)$};
\node [right] at (1.5,0.3) {$x$};
\node [left] at (-1.2,3.14) {$\pi$};
\node[align=left, below] at (0,-0.3)%
{$d=d_0$};
\end{tikzpicture}\ \ \ \
\begin{tikzpicture}[yscale=0.7]
\draw [->] (-1.5,0.3) -- (1.5,0.3);
\draw [->] (0,0) -- (0,4);
\draw [thick,domain=-1:1] plot (\x, {1/180*pi*atan(20*tan(\x r))+0.5*pi-0.5*(\x-0.4)});
\draw [semithick,domain=-1.2:1.2] plot (\x, {0*\x+pi});
\node [above] at (0,4) {$f(x)$};
\node [right] at (1.5,0.3) {$x$};
\node [left] at (-1.2,3.14) {$\pi$};
\node[align=left, below] at (0,-0.3)%
{$0\le d<d_0$};
\end{tikzpicture}
\caption{Bifurcation between $\CU\CH$ and $\CN\CU\CH$}
\label{f.bifurcation}
\end{center}
\end{figure}

Note by (\ref{g1}), we have $g_1:\R/\Z\times J\rightarrow \R\PP^1$ with $g_1(x,t)=\tan^{-1}[t-v(x)]$. Set $I_{0,j}(t)=\R/\Z$ for all $t\in \CI$ and $j=1,2$. Then inductively, we define the following for each $i\ge 1$.
\begin{itemize}
\item $i$th step \emph{critical points}:
$$
C_i(t)=\{c_{i,1}(t), c_{i,2}(t)\}
$$
with $c_{i,j}(t)\in I_{i-1,j}(t)$ minimizing $\{|g_i(x,t)|,\ x\in I_{i-1,j}(t)\}$. More precise description of $C_i(t)$ will be given in Theorem~\ref{t.iteration}.

\vskip0.2cm
\item $i$th step \emph{critical interval}:
$$
I_{i,j}(t)=\{x:|x-c_{i,j}(t)|\le \frac{1}{2^iq_{N+i-1}^{2\tau}}\},\ I_i(t)=I_{i,1}(t)\cup I_{i,2}(t).
$$

\item $i$th step \emph{return times}:

$$
q_{N+i-1}\le r^{\pm}_i(x, t):I_i(t)\rightarrow\Z^+
$$
is the first return times (back to $I_i(t)$) after time $q_{N+i-1}-1$. Here $r^+_i(x, t)$ is the forward return time and $r^-_i(x, t)$ backward. Let $r_i=\min\{r^+_{i},r^-_{i}\}$ with $r^{\pm}_i=\min_{x\in I_i(t), t\in J}r^{\pm}_i(x, t)$.

\vskip0.2cm
\item Then $i+1$-th step $g_{i+1}$:
 $$
 g_{i+1}(x,t)=s_{r^+_i}(x,t)-u_{r^-_i}(x,t):D_i\rightarrow\R\PP^1,
 $$
 where we define
 $$
 D_i:=\{(x,t):x\in I_{i}(t),\ t\in J\}.
 $$
\end{itemize}

Then the following theorem is from the Induction Theorem (Theorem 3) of \cite{Wang-Zhang}.

\begin{theorem}\label{t.iteration}
For any $\e>0$, there exists a $\l_0=\l_0(v,\a,\e)$ such that for all $\l>\l_0$, the following holds for each $i\ge 2$.
\begin{itemize}

\item  For each $t\in J$, $g_{i}(\cdot, t):I_{i-1}(t)\rightarrow\mathbb R\mathbb P^1$ is of types $\I$, $\II$ or $\III$, which are denoted as case $(i)_\I$, $(i)_\II$ and $(i)_\III$. In cases $(i)_\I$ and $(i)_\II$, as
          function on $I_{i-1}(t)$,
          \beq\label{gi+2-closeto-gi+1}
          \|g_{i}-g_{i-1}\|_{C^2}\le C\l^{-\frac32 r_{i-1}}.
          \eeq
          Moreover, we have the following.
          \begin{itemize}
          \item In case $(i)_\I$ and $(i)_\III$, $I_{i-1,1}(t)\cap I_{i-1,2}(t)=\varnothing$. So we need to consider
          $g_i:I_{i-1,j}(t)\rightarrow\R\PP^1$ respectively for $j=1,2$. Then in case $(i)_\I$, if $g_{i}(\cdot,t)$
          is of type $\I_+$ on $I_{i-1,1}(t)$, then it is of type $\I_-$ on $I_{i-1,2}(t)$, vice versa.
          \vskip0.2cm
          \item In case $(i)_\II$, $I_{i-1,1}(t)\cap I_{i-1,2}(t)\neq \varnothing$. So $g_i$ is of type $\II$ as
          function on the connected interval $I_{i-1}(t)$.
          \end{itemize}
          \vskip0.1cm
          Thus it's clear from Definition~\ref{d.type-of-functions} that in case $(i)_\I$ and $(i)_\II$, $c_{i,j}(t)$ is the only point minimizing $g_i$ on
          $I_{i-1,j}$.  Note in case $(i)_\II$, it's possible that $c_{i,1}(t)=c_{i,2}(t)$. In case $(i)_\III$, however, we may have more choices on each of $I_{i-1,j}(t)$ for $j=1,2$. Comparing Lemma~\ref{l.s-deri-resonance} and figure 2 in case of $d>d_0$, we choose the corresponding $x_2$ of $g_i: I_{i-1,j}(t)\rightarrow\R\PP^1$ as our $c_{i,j}(t)$.

\vskip0.2cm
\item For each $i\ge 1$ and $t\in J$, it holds that
         \beq\label{ci-close-ci+1}
         |c_{i-1,j}(t)-c_{i,j}(t)|<C\l^{-\frac34r_{i-2}},\ j=1,2;
         \eeq

\item For all $x\in I_{i-1}(t)$ and $m=1,2$, it holds that
\begin{align}
\label{norm-i+2}&\|A_{\pm r^{\pm}_{i-1}(x, t)}(x,t)\|>\l^{(1-\e)r^{\pm}_{i-1}(x, t)}\ge \l^{(1-\e)q_{N+i-2}};\\
\label{norm-i+2.2}&\frac{\partial^m(\|A_{\pm r^{\pm}_{i-1}}(x,t)\|)}{\partial \nu^m}<\|A_{\pm r^{\pm}_{i-1}}(x,t)\|^{1+\e},\ \nu=x\mbox{ or }t.
\end{align}

\item In case $(i)_\III$, there exists a unique $k$ such that $1\le |k|< q_{N+i-2}$ and
      $$
      I_{i-1,2}\cap \left(I_{i-1,1}+k\a\right)\neq\varnothing.
      $$
      Moreover, there exists points $d_{i,j}(t)\in I_{i-1,j}(t)$ such that
      $$
      g_{i}(d_{i,j}(t),t)=g_{i}(c_{i,j}(t),t),\ j=,1,2,
      $$
      and the following hold.

     \begin{itemize}
      \item if $|g_{i}(c_{i,j}(t),t)|>c\l^{-\frac1{10}r_{i-1}}$, $j=1$ or $2$, then so are $|g_{i+1}(c_{i+1,j}(t),t)|$
      for $j=1$ and $2$;
      \item if $|g_{i}(c_{i,j}(t),t)|<C\l^{-\frac1{10}r_{i-1}}$, $j=1$ or $j=2$, then there exists a  such that
            \beq\label{i+2-orbit-critical-points}
            \hskip 1cm |c_{i,1}(t)+k\a-d_{i,2}(t)|,\ |c_{i,2}(t)-k\a-d_{i,1}(t)|<C\l^{-\frac1{30} r_{i-1}}.
            \eeq
     \end{itemize}
 (\ref{i+2-orbit-critical-points}) implies that as $i$-th step `critical points', $d_{i,1}(t)$ is essentially the $-k$-orbit of $c_{i,2}(t)$ while $d_{i,2}(t)$ is the $k$-orbit of $c_{i,1}(t)$. Thus, we change the notation as
 \beq\label{k-orbits}
 d_{i,1}(t)=c^{-k}_{i,2}(t),\ d_{i,2}(t)=c^k_{i,1}(t).
 \eeq
\end{itemize}
\end{theorem}

\begin{remark}[Idea of Proof of Theorem~\ref{t.iteration}]\label{r.iteration}
For simplicity, we leave $t$ implicit in $I_{i,j}(t)$ or $I_{i}(t)$ in this remake.

The proof of Theorem~\ref{t.iteration} in \cite{Wang-Zhang} is inductive, and contains the following different cases. For simplicity, let
\beq\label{all-possibilities}
(i)_a\rightarrow (i+1)_b
\eeq
denotes the ``from case $(i)_a$ to case $(i+1)_b$'' where $a$, $b=\I$, $\II$ or $\III$. Then \cite{Wang-Zhang} showed that start from cases $(i)_a$ with $a=\I,\II,\III$, one necessarily falls into cases $(i+1)_b$ with $b=\I,\II,\III$. Then, in (\ref{all-possibilities}),
\vskip0.2cm
\begin{center}
if $a$ is $\I$ or $\III$, then $b$ can and only can be $\I$ or $\III$; \\ if $a$ is $\II$, then $b$ can be $\I$, $\II$ or $\III$.
\end{center}
In particular, only case $(i)_\II$ can lead to case $(i+1)_\II$. Moreover, whenever we end with the case $(i+1)_\I$ or $(i+1)_\II$, or in the case $(i)_\III\rightarrow (i+1)_\III$, we must have as functions defined on $I_{i}$,
\beq\label{gi+1-close-to-gi}
\|g_{i+1}-g_i\|_{C^2}< C\l^{-\frac32r_{i-1}}<C\l^{-\frac32q_{N+i-2}}.
\eeq
So basically, the essential change only happens in cases $(i)_\I\rightarrow(i+1)_\III$ and $(i)_\II\rightarrow(i+1)_\III$, which are essential in the same situation. Indeed, for the case $(i)_\II\rightarrow (i+1)_\III$, $I_{i-1}$ will be splitted into two disjoint interval $I_{i,1}$ and $I_{i+1,2}$. And on each interval$I_{i,j}$, $g_i$ will be of type $\I$. One is of type $\I_-$ and the other $\I_+$. More concretely, for these two cases, we have the following.

Let $r_{i-1}\le |k|<q_{N+i-1}$ be the time such that $I_{i,2}\pm (I_{i,1}+k\a)\neq\varnothing$. Then we define
$$
g'_{i,1}=s_k-u_{r^-_{i}}:I_{i,1}\rightarrow\mathbb R\mathbb P^1,\ g'_{i,2}=s_{r^+_{i}-k}-u_k:I_{i,2}\rightarrow\mathbb R\mathbb P^1
$$
 which satisfy
\beq\label{i+1r-1}
\left\|g'_{i,j}-g_{i}\right\|_{C^2}<C\l^{-\frac 32r_{i-1}}<C\l^{-\frac32 q_{N+i-2}}.
\eeq
This implies that $g'_{i,j}$ is of the same type as $g_{i}: I_{i,j}\rightarrow\R\PP^1$. Then to get all the necessary estimate of the geometric properties of $g_{i+1}$, it is sufficient to take
\beq\label{g-i+1}
g_{i+1}(x,t)=\begin{cases}\tan^{-1}(l_k^2\tan [g'_{i,2}(x+k\a,t)])-\frac\pi2+g'_{i,1}(x,t),\ &x\in I_{i,1},\\ \tan^{-1}(l_k^2\tan [ g'_{i,1}(x-k\a,t)])-\frac\pi2+g'_{i,2}(x,t),\ &x\in I_{i,2}\end{cases}
\eeq
where $l_k>c\l^{0.9k}>c\l^{0.9r_{i-1}}\ge c\l^{0.9q_{N+i-2}}$. This yields type $\III$ functions, hence, case $(i+1)_\III$.
\end{remark}

\begin{remark}[Derivative estimates over paramter $t$]\label{r.t-estimate} In Theorem~\ref{t.iteration}, all the derivative estimates of $g_i(x,t)$ and $\|A_{r_i}(x,t)\|$ are for $x$. However, all the necessary technical lemmas in \cite{Wang-Zhang} (Lemma 2--5) can be directly applied to the derivative estimates over parameter $t$.

On the other hand, by Theorem~\ref{t.iteration}, for each $t\in J$, we have either $g_i(c_{i,j},t)=0$ and $\frac{\partial g_i}{\partial x}(c_{i,j},t)\neq 0$, or $\frac{\partial g_i}{\partial x}(c_{i,j},t)=0$ and $\frac{\partial^2 g_i}{\partial x^2}(c_{i,j},t)\neq 0$. Thus, $c_{i,j}(t)$ are functions in $t$ and vary continuously over $t$. Thus if for any fixed $t_0$, $g_i(\cdot,t_0)$ is of type $a$, $a=\I,\II,\III$, then for all $t$ sufficiently close to $t_0$, $g_i(\cdot,t)$ are also of type $a$.
\end{remark}

\section{Proof of Main Theorem}\label{s.proof-of-main-theorem}
In this section, based on the induction Theorem~\ref{t.iteration}, we are going to prove Theorem B. First note for each $i\ge 0$,
$$
D_i:=\{(x,t):x\in I_{i}(t),\ t\in J\}.
$$

Then we let $c^0_{i,j}(t)=c_{i,j}(t)$. Recall by (\ref{k-orbits}), for $0\neq |k|<q_{N+i-2}$, $c^k_{i,j}(t)$'s are the extra minimal points of $g_{i}(\cdot,t)$. And as minimal points of $g_i(\cdot,t)$, they essentially lie on the $k$-orbit of $c_{i,j}(t)$.

Recall for a connected interval $S\subset\T$ and constant $0<a\le1$, $aS$ is defined to be the subinterval of $S$ with the same center and whose length is $a|S|$.
\begin{lemma}\label{l.key}
For each $i\ge 1$, consider $g_i:D_{i-1}\rightarrow\R\PP^1$. Then we have the following estimates
\begin{align}
\label{devi-estimate-i.1} &\left|\frac{\partial g_i}{\partial x}(c_{i,j}(t),t)\right|<C,\ j=1,2;\\
\label{devi-esitmate-i.1.5}&\prod_{j=1,2}\frac{\partial g_i}{\partial x}(c_{i,j}(t),t)\le 0;\\
\label{devi-estimate-i.2} &\frac{\partial g_i}{\partial t}(x,t)>c,\ (x,t)\in D_{i-1}.
\end{align}
Here $\frac{\partial g_i}{\partial x}(c_{i,j}(t),t)=0$ if and only if $c_{i,j}(t)=c^k_{i,j'}(t)$ for some $|k|<q_{N+i-2}$. Here $j\neq j'\in\{1,2\}$. Moreover, if $c_{i,1}(t)=c^{-k}_{i,2}(t)$, then it holds that
\beq\label{devi-estimate-i.3}
\frac{\partial g_i}{\partial x}(c_{i,1}(t),t)=0\mbox{ and }\left|\frac{\partial^2 g_i}{\partial x^2}(c_{i,1}(t),t)\right|>c;
\eeq
If $c_{i,2}(t)=c^{k}_{i,1}(t)$, then we have
\beq\label{devi-estimate-i.4}
\frac{\partial g_i}{\partial x}(c_{i,2}(t),t)=0\mbox{ and }\left|\frac{\partial^2 g_i}{\partial x^2}(c_{i,2}(t),t)\right|>c.
\eeq

\end{lemma}

\begin{proof}
(\ref{devi-estimate-i.3}) and (\ref{devi-estimate-i.4}) follows directly from either the definition of type $\II$ function or definition of type $\III$ function and (\ref{type3-nondegeneracy}). Thus, we focus on (\ref{devi-estimate-i.1})--(\ref{devi-estimate-i.2}). We are going to proceed by some sort of induction. For $i=1$, since
$$
g_1(x,t)=s_1(x,t)-u_1(x,t)=\tan^{-1}[t-v(x)]:\R/\Z\times J\rightarrow\R\PP^1,
$$
it's clearly that for all $(x,t)\in \R/\Z\times J=D_0$,
\begin{align}
\label{devi-estimate-1.1} &\frac{\partial g_1}{\partial t}(x,t)>c\\
\label{devi-estimate-1.2} &\left|\frac{\partial g_1}{\partial x}(x,t)\right|<C.
\end{align}

Now we consider $g_{i}: D_{i-1}\rightarrow\R\PP^1$ with $i\ge 2$. Then for each $t\in J$, $g_{i}(\cdot,t):I_{i-1}(t)\rightarrow\R\PP^1$ falls into one of the cases $(i)_{\I}$, $(i)_\II$ or $(i)_\III$. Throughout this proof, we may fix a $t_0$ and consider a sufficiently small interval $J_0$ around $t_0$. Then by Remark~\ref{r.t-estimate}, we may assume that for all $t\in J_0 $, $g_i(\cdot,t): I_{i-1}(t)\rightarrow\R\PP^1$ are of the same type as $g_{t_0}$. Now we do the following induction. For simplicity, we define for $0<a\le 1$
$$
aD'_{i,j}=\{(x,t):x\in aI_{i,j}(t), t\in J_0\},\ aD'_{i}=aD'_{i,1}\bigcup aD'_{i,2}
$$

\begin{itemize}
\item If $g_i(\cdot,t_0)$ is of type $(i)_\I$, then we assume the following.
      \begin{align}
       \label{devi-estimate-i.5} &\left|\frac{\partial g_i}{\partial x}(x,t)\right|<C,\ \forall (x,t)\in\frac{D'_{i-1}}{2};\\ \label{devi-esitmate-i.5.5}&\prod_{j=1,2}\frac{\partial g_i}{\partial x}(x^j,t)<0,\ \forall (x^j,t)\in\frac{D'_{i-1,j}}{2}\\
       \label{devi-estimate-i.6} &\frac{\partial g_i}{\partial t}(x,t)>c,\ \forall (x,t)\in D'_{i-1}.
      \end{align}
\item If $g_i(x,t_0)$ is of type $(i)_\II$, then we assume as function on $D'_{i-1}$,
      \beq\label{devi-estimate-i.7}
      \left\|g_i-g_{i-1}\right\|_{C^2}<C\l^{-\frac32q_{N+i-3}}.
      \eeq
      Here we set $q_{N-1}=1$.
\item If $g_i(x,t_0)$ if of type $(i)_\III$, then there exist a $1\le |k|\le q_{N+i-2}$ such that
      $$
      |c_{i-1,1}-c_{i-1,2}-k\a|<\frac{|I_{i-1,1}|}{2}.
      $$
      Let $d_{i-1}=c_{i-1,1}-c_{i-1,2}-k\a$. Then if $d_{i-1}>\frac{|I_{i-1,1}|}{4}$, we assume (\ref{devi-estimate-i.5}) and (\ref{devi-estimate-i.6}) hold. If $d_{i-1}\le \frac{|I_{i-1,1}|}{4}$, then we do the following. Let $z_{i,j}(t)\in I_{i-1,j}(t)$ be one of the zeros of $\frac{\partial g_i}{\partial x}(\cdot,t)$ on $I_{i-1,j}(t)$ that sits between $c_{i,j}(t)$ and $c^{\pm k}_{i,j'}(t)$ where $j\neq j'\in\{1,2\}$. So we may let
      $$
      \hskip 1cm I_{i-1,j}(t)=I^{(1)}_{i-1,j}(t)\bigcup I^{(2)}_{i-1,j}(t)\mbox{ with }I^{(1)}_{i-1,j}(t)\bigcap I^{(2)}_{i-1,j}(t)=\{z_{i,j}(t)\}.
      $$
      Let $I^{(1)}_{i-1,j}(t)$ be the one containing $c_{i,j}(t)$ and $I^{(1)}_{i-1}(t)=\bigcup_{j=1,2}I^{(1)}_{i-1,j}(t)$. Then we assume the following.
      \begin{align}
       \label{devi-estimate-i.8} &\left|\frac{\partial g_i}{\partial x}(x,t)\right|<C,\ \forall t\in J_0,\ \forall x\in I^{(1)}_{i-1}(t);\\ \label{devi-estimate-i.8.5}&\prod_{j=1,2}\frac{\partial g_i}{\partial x}(x_j,t)\le 0,\forall t\in J_0,\ x_j\in I^{(1)}_{i-1,j}(t);\\
       \label{devi-estimate-i.9} &\frac{\partial g_i}{\partial t}(x,t)>c,\ \forall (x,t)\in D'_{i-1}.
      \end{align}
\end{itemize}

Given (\ref{devi-estimate-i.5})--(\ref{devi-estimate-i.9}), we clearly have (\ref{devi-estimate-i.1}) and (\ref{devi-estimate-i.2}). Indeed, (\ref{devi-estimate-i.5})--(\ref{devi-estimate-i.6}) of type $\I$ and (\ref{devi-estimate-i.8})--(\ref{devi-estimate-i.9}) of type $\III$ directly imply (\ref{devi-estimate-i.1})--(\ref{devi-estimate-i.2}). In case of type $\II$, by Remark~\ref{r.iteration}, it holds that
$$
\|g_i-g_1\|_{C^2}<C\sum_{j=1}^{i-1}\|g_{i+1}-g_{i}\|<C\sum_{j=1}^{i-1}C\l^{-\frac32 q_{N+j-2}}<C\l^{-\frac32},
$$
which together with (\ref{devi-estimate-1.1}) and (\ref{devi-estimate-1.2}) clearly imply (\ref{devi-estimate-i.1})--(\ref{devi-estimate-i.2}).

\vskip 0.2cm
Now we move the induction to the step-$(i+1)$. Consider $g_i: D'_i\rightarrow\R\PP^1$ which clearly satisfies all the assumption of $g_{i}:D'_{i-1}\rightarrow\R\PP^1$. Then by Remark~\ref{r.iteration}, if we are in one of the cases
$$
(i)_\I,\ (i)_\II,\ (i)_\III\rightarrow (i+1)_\I,\ (i)_\II\rightarrow (i+1)_\II, \mbox{ or } (i)_\III\rightarrow (i+1)_\III
$$
then it holds that as functions on $D'_i$,
$$
\|g_{i+1}-g_{i}\|_{C^2}<C\l^{-\frac 32 q_{N+i-2}}.
$$
Then all the induction assumptions for step-$(i+1)$ naturally follows from those of step-$(i)$.

Thus we focus on the case $(i)_\I, (i)_\II\rightarrow(i+1)_\III$. Again by Remark~\ref{r.iteration}, the cases $(i)_\I, (i)_\II\rightarrow (i+1)_\III$ can be done in the same way, so we consider $(i)_\I\rightarrow(i+1)_\III$. Then there exists a \emph{resonance time} $r_{i-1}\le |k|\le q_{N+i-1}-1$ such that
$$
d_i:=|c_{i,1}-c_{i,2}-k\a|<\frac{|I_{i,1}|}{2}.
$$

Then we define $g'_{i,1}=s_k-u_{r^-_{i}}:D'_{i,1}\rightarrow\mathbb R\mathbb P^1$ and $g'_{i,2}=s_{r^+_{i}-k}-u_k:D'_{i,2}\rightarrow\mathbb R\mathbb P^1$ which satisfy
\beq\label{i+1r-1}
\left\|g'_{i,j}-g_{i}\right\|_{C^2}<C\l^{-\frac 32r_{i-1}}<C\l^{-\frac32 q_{N+i-2}}.
\eeq
This implies that $g'_{i,j}$ satisfies all the assumptions of $g_{i}: D'_{i,j}\rightarrow\R\PP^1$. Then to estimate the geometric properties of $g_{i+1}$, it is sufficient to take
\beq\label{g-i+1}
g_{i+1}(x,t)=\begin{cases}\tan^{-1}(l_k^2\tan [g'_{i,2}(x+k\a,t)])-\frac\pi2+g'_{i,1}(x,t),\ &x\in D'_{i,1},\\ \tan^{-1}(l_k^2\tan [ g'_{i,1}(x-k\a,t)])-\frac\pi2+g'_{i,2}(x,t),\ &x\in D'_{i,2}\end{cases}
\eeq
where $l_k>c\l^{0.9k}>c\l^{0.9r_{i-1}}\ge c\l^{0.9q_{N+i-2}}$.

Now if $d_i>\frac{I_{i,1}}{4}$, by the proof of \cite[Lemma 6]{Wang-Zhang}, or directly from (\ref{g-i+1}), it's straightforward calculation to see that
$$
\left\|g_{i+1}-g'_{i,j}\right\|_{C^2}<Cl_k^{-\frac32}<C\l^{-\frac32q_{N+i-2}}.
$$
Thus we get as function on $\frac{D'_{i,j}}{2}$,
$$
\left\|g_{i+1}-g_{i}\right\|_{C^2}<\left\|g_{i+1}-g'_{i,j}\right\|_{C^2}+\left\|g'_{i,j}-g_{i}\right\|_{C^2}<C\l^{-\frac32q_{N+i-2}},
$$
which together with step-$(i)$ induction assumption in case of $(i)_\I$ clearly implies the corresponding step-$(i+1)$ induction assumption.

If $d<\frac{|I_{i,1}|}{4}$, we focus on the interval $I_{i,1}(t)$ since then the case on $I_{i,2}(t)$ can be done in the same way. By (\ref{i+1r-1}) and our definition of type $\I$ function, we clear get that on $D'_{i,1}$,
\beq\label{i+1r-1.1}
0<\frac{\partial g'_{i,1}(x,t)}{\partial x}<C,\ -C<\frac{\partial g'_{i,2}(x+k\a,t)}{\partial x}<0,\ \frac{\partial g'_{i,j}(x,t)}{\partial t}>c.
\eeq

Thus, we may let $z_{i+1,1}(t)\in I_{i,1}(t)$ be one of the zeros of $\frac{\partial g_{i+1}(\cdot,t)}{\partial x}$ on $I_{i,1}(t)$ that sits between $c_{i+1,1}(t)$ and $c^{-k}_{i+1,2}(t)$. And set
      $$
      I_{i,1}=I^{(1)}_{i,1}(t)\bigcup I^{(2)}_{i,1}(t)\mbox{ with }I^{(1)}_{i,1}(t)\bigcap I^{(2)}_{i,1}(t)=\{z_{i+1,1}(t)\}.
      $$
Let $I^{(1)}_{i,1}(t)$ be the one always containing $c_{i+1,1}(t)$.

Then (\ref{i+1r-1}) and (\ref{i+1r-1.1}) together clearly implies that on $\{(x,t)\in I^{(1)}_{i,1}(t),\ t\in J_0\}$,
\beq
0\le \frac{\partial g_{i+1}(x,t)}{\partial x}<\frac{\partial g'_{i,1}(x,t)}{\partial x}<C,\ \frac{\partial g_{i+1}(x,t)}{\partial t}>\frac{\partial g'_{i,1}(x,t)}{\partial t}>c,
\eeq
which clearly implies our induction assumption at step-$(i+1)$.
\end{proof}

Now we define the function $\rho_{i}: J\rightarrow\R\PP^1=\R/(\pi\Z)$ such that
\beq\label{dist-cri-pts}
\rho_{i}(t)=\begin{cases}g_i(c_{i,1}(t),t), &\mbox{ if }c_{i,1}(t)=c^k_{i,2}(t)\mbox{ for some }k\in\Z.\\ c_{i,1}(t)-c_{i,2}(t), &\mbox{ otherwise.}\end{cases}
\eeq

Define $\mathcal Z_i=\{t: \rho_i(t)=0\}$. By the definition (\ref{dist-cri-pts}) of $\rho_i$, we clearly have
\beq\label{zero-of-dist-fun}
 \CZ_i=\{t:g_i(c_{i,1}(t),t)=\frac{\partial g_i}{\partial x}(c_{i,1}(t),t)=0\}.
\eeq

By (\ref{devi-estimate-i.3}) and (\ref{devi-estimate-i.4}), for each $t\in\CZ_i$, $g_i(c_{i,1}(t),t)$ is the local extremal. For each $t$, let $h$ be the critical point of $g_i(x,t)$. We know that $\frac{d^2g_i}{dx^2}(h)\neq 0$. Thus, by Implicit Function Theorem, $h$ is a function of $t$. We may assume $h(t)$ is such that $h(t)=c_{i,1}(t)$ for $t\in\CZ_i$. Consider the function $e(t)=g(h(t),t)$. Then we get
\beq\label{local-extremal-devi}
\frac{de(t)}{dt}=\frac{\partial g_i}{\partial x}(h(t),t)\cdot\frac{dh}{dt}(t)+\frac{\partial g_i}{\partial t}(h(t),t)=\frac{\partial g_i}{\partial t}(h(t),t)>0,
\eeq
which clearly implies that $\CZ_i$ is an isolated set. By choosing $J$ compact, we clearly get that the finiteness of $\CZ_i$. Then we have the following corollary of Lemma~\ref{l.key}.

\begin{corollary}\label{c.key}
For the $\rho_{i}$ above, it holds that
\beq
\left|\frac{d\rho_i}{dt}(t)\right|>c,\ \forall t\notin\CZ_i
\eeq
\end{corollary}

\begin{proof}
First consider the case $\rho_i(t)=c_{i,1}(t)-c_{i,2}(t)$. Then by Lemma~\ref{l.key}, we have that
$$
\left|\frac{dc_{i,j}(t)}{dt}\right|=\left|\frac{\partial g_i(c_{i,1}(t),t)}{\partial t}/\frac{\partial g_i(c_{i,1}(t),t)}{\partial x}\right|>c
$$
and
$$
\frac{dc_{i,j}(t)}{dt}\cdot\frac{dc_{i,j}(t)}{dt}<0.
$$
Thus, we must have
$$
\left|\frac{d\rho_i(t)}{dt}\right|=\left|\frac{dc_{i,1}(t)}{dt}-\frac{dc_{i,2}(t)}{dt}\right|>c.
$$

If $\rho_i(t)=g_{c_{i,1}(t),t}$, then by definition, $\rho_i(t)$ is the local extremal of $g_i(x,t)$ for each $t$. Thus similar to (\ref{local-extremal-devi}), we get

\begin{align*}
\left|\frac{d\rho_i(t)}{dt}\right| &=\left|\frac{\partial g_i}{\partial x}(c_{i,1}(t),t)\cdot\frac{dc_{i,1}}{dt}(t)+\frac{\partial g_i}{\partial t}(c_{i,1}(t),t)\right|\\ &=\left|\frac{\partial g_i}{\partial t}(c_{i,1}(t),t)\right|\\ &>c.
\end{align*}
\end{proof}

Next, we have the following Lemma.

\begin{lemma}\label{l.large-gap}
There exist a $n\in \Z_+$ such that for some $t\in J$,
\beq\label{large-gap}
\left|\rho_n(t)\right|=\left|g_n(c_{n,1}(t),t)\right|>\l^{-\frac{1}{10}r_{n-1}}.
\eeq
\end{lemma}

\begin{proof}
 We first consider the case that for some $i\ge 1$, $\CZ_i$ contains at least two points. Then, by Corollary~\ref{c.key}, there exists a subinterval $J'\subset J$ such that either
$$
\mathrm{Leb}\{g_i(c_{i,1}(t),t):t\in J'\}=\pi\mbox{ or }
$$
$$
\mathrm{Leb}\{c_{i,1}(t)-c_{i,2}(t):t\in J'\}=\pi.
$$
In the first case, we clearly get (\ref{large-gap}) for $n=i$ since $\l\gg1$. In the latter case, it's also clear that for some $r_{i-1}\le |k_i|<q_{N+i-1}$ and for some large $i$ that
$$
\{t:\left\|c_{i,1}(t)-c_{i,2}(t)-k_i\a\right\|_{\R/\Z}<C\l^{-\frac12r_{i-1}}\}\subset J'.
$$
This implies the occurrence of strongest resonance at the next step. More precisely, we have two functions $g'_{i,j}: I_{i,j}\rightarrow\R\PP^1$, $j=1,2$, satisfying
\begin{align}
\label{large-gap1} &\|g'_{i,j}-g_i\|_{C^2}<C\l^{-\frac32 r_{i-1}};\\
\label{large-gap2} &g_{i+1}(x,t)=\tan^{-1}(l_{k_i}^2\tan [g'_{i,2}(x+k_i\a,t)])-\frac\pi2+g'_{i,1}(x,t),
\end{align}
where $\l^{\frac34 r_{i-1}}<l_{k_i}<\l^{r_i^{\frac34}}$. Let $\bar c_{i,j}(t)$ be the zero of $g'_{i,j}(x,t)$. Then by (\ref{large-gap1}), it clearly holds that
$$
\left|\bar c_{i,j}(t)-\bar c_{i,j}(t)\right|<C\l^{-\frac32 r_{i-1}}.
$$
Hence, we get
$$
\{t:\left\|\bar c_{i,1}(t)-\bar c_{i,2}(t)-k_i\a\right\|_{\R/\Z}<C\l^{-\frac12r_{i-1}}\}\subset J'.
$$
In particular, there exists $t_0\in J'\subset J$ such that $\bar c_{i,1}(t_0)+k_i\a=\bar c_{i,2}(t_0)$. Together with (\ref{large-gap2}), it's easy to see that
$$
\min_{x\in I_{i}(t_0)}\left|g_{i+1}(x,t_0)\right|>cl^{-\frac32}_{k_i}>c\l^{-\frac32r_i^{\frac34}}>c\l^{-\frac1{10}r_i}.
$$
This clearly implies (\ref{large-gap}) for $n=i+1$.

\vskip 0.2cm
Next we assume that for all $i\ge 1$, $\CZ_i$ contains at most one point, set $\CZ_i=\{t_i\}$. Thus we may let
$$
\rho_i(t)=\begin{cases}g_i(c_{i,1}(t),t), &t\in J_1;\\ c_{i,1}(t)-c_{i,2}(t), &t\in J_2,\end{cases}
$$
where $J_1\cup J_2\cup\{t_i\}=J$. Then by Corollary~\ref{c.key}, we have that
$$
\mathrm{Leb}\{g_i(c_{i,1}(t),t):t\in J_1\}+\mathrm{Leb}\{c_{i,1}(t)-c_{i,2}(t):t\in J_2\}>c|J|
$$
Then either for some large $i$,
$$
\mathrm{Leb}\{g_i(c_{i,1}(t),t):t\in J_1\}>c|J|>c\l^{-\frac1{10}r_{i-1}}
$$
which is nothing other (\ref{large-gap}) for $n=i$. Or for all large $i$, we have
$$
\mathrm{Leb}\{c_{i,1}(t)-c_{i,2}(t):t\in J_2\}>c|J|.
$$
By (\ref{ci-close-ci+1}), we may choose large $M$ so that
$$
|c_{i,j}(t)-c_{i+1,j}(t)|\ll c|J|
$$
for all $i\ge M$. In other words, we may assume $\{c_{i,1}(t)-c_{i,2}(t):t\in J_2\}$ is a fixed interval on $\R\PP^1$ for large $i$. Thus, by density of the set $\{k\a, k\in \Z\}$ on $\R/\Z$, we must have that for some large $i$ and some $r_{i-1}\le k_i<q_{N+i-1}$,
$$
\{t:\left\|c_{i,1}(t)-c_{i,2}(t)-k_i\a\right\|_{\R/\Z}<C\l^{-\frac12r_{i-1}}\}\subset J_2.
$$
Then we may proceed as before and conclude (\ref{large-gap}) for $n=i+1$.
\end{proof}

Now we are ready to prove Theorem~B. Let's first state a proposition.
\begin{prop}\label{p.large-norm}
Let $\{B^{(k)}\}_{k\in\Z}\subset\mathrm{SL}(2,\R)$ be a bounded sequence. In other words, $\|B^{(k)}\|<C$ for all $k\in\Z$. Let $\beta=\min_{k\in\Z}\|B^{(k)}\|$. Assume the following holds.
$$
\gamma:=\inf_{k\in\Z}\left|\tan[s(B^{(k)})-u(B^{(k-1)})]\right|>\frac{2}{\beta-\beta^{-1}},
$$
Then there exists a $\rho>1$ such that for each $k\in Z$ and each $n\ge1$
$$
\|B^{(k+n-1)}\cdots B^{(k)}\|>\rho^n.
$$
Moreover if $\beta\gg\frac1\gamma\gg 1$, then $\rho>c\beta\gamma\gg 1$.
\end{prop}
See e.g. \cite[Lemma 11]{zhang} for a proof.
\begin{proof}[\textbf{Proof of Theorem B}]
For the arbitrary given interval $S\subset [\inf v-2, \sup v+2]$, by Lemma~\ref{l.large-gap}, there exists a $t_0\in J$ such that for some large $n\in\Z_+$,
$$
\min_{x\in I_{n-1,1}(t_0)}\left|g_{n}(x,t_0)\right|>c\l^{-\frac1{10}r_{n-1}}.
$$

By \cite[Theorem 3]{Wang-Zhang}, the above inequality implies that
\beq\label{proof-of-theorem1.1}
\min_{x\in I_{n-1}(t_0)}\left|g_{n}(x,t_0)\right|>c\l^{-\frac1{10}r_{n-1}}.
\eeq

On the other hand, we know from \cite[Theorem 3]{Wang-Zhang} that
\beq\label{proof-of-theorem1.2}
\|A_{r^+_{n-1}(x, t_0)}(x,{t_0})\|>\l^{\frac9{10}r^+_{n-1}(x, t_0)},\ \forall x\in I_{n-1}(t_0).
\eeq

By \emph{Diophantine} condition, there exists a $M_1$ polynomially large in $|I_{n-1}(t_0)|^{-1}$ such that for each $x\in\R/\Z$, $x+m\a\in I_{n-1}(t_0)$ for some $0\le m\le M_1$ (see e.g. \cite[Lemma 6]{aviladamanikzhang}).

Set $M_2=\max_{x\in I_{n-1}}[r^{\pm}_{n-1}(x, t_0)]^2$. Then for each $x\in I_{n-1}(t_0)$ and for each $M\ge M_2$,  let $1\le j_p\le M,\ 1\le p\le m$ be all the times such that
 $$
 j_{p}-j_{p-1}\ge q_{N+n-2},\ x+j_p\a\in I_{n-1}(t_0),
 $$
where we set $j_0=0$. Also it's clear that $M-j_m<M^{\frac12}$. Then we have
\begin{align*}
\|A_M(x, t_0)\|&=\left\|A_{M-j_m}(x+j_m\a, t_0)\cdot\prod^{m}_{p=1}A_{j_p-j_{p-1}}(x+j_{p-1}\a, t_0)\right\|\\
&\ge \left\|A_{M-j_m}(x+j_m\a, t_0)\right\|^{-1}\cdot\left(\prod^{m}_{p=1}\left\|A_{j_p-j_{p-1}}(x+j_{p-1}\a, t_0)\right\|\right)^{\frac35}\\
&\ge \l^{-M^{\frac12}}\cdot\l^{\frac12 j_m}\\
&\ge \l^{\frac12 M-2M^{\frac12}}\\
&\ge \l^{\frac13 M},
\end{align*}
where the second estimate follows from (\ref{proof-of-theorem1.1}), (\ref{proof-of-theorem1.2}) and Proposition~\ref{p.large-norm}. Indeed, it clearly holds for each $1\le p\le m$,
$$
1\gg \left|[s_{r_n^+(x+j_p\a, t_0)}-u_{r_n^-(x+j_p\a, t_0)}](x+j_p\a, t_0)\right|>c|g_{n}(x,t_0)|\gg \|A_{j_{p+1}-j_p}(x+j_p, t_0)\|^{-1}.
$$
Thus, Proposition~\ref{p.large-norm} can be applied to the finite sequence
$$
A_{j_m-j_{m-1}}(x+j_{m-1}\a, t_0),\ldots,A_{j_p-j_{p-1}}(x+j_{p-1}\a, t_0),\ldots,A_{j_1}(x, t_0)
$$
and yields the estimate above.

Now for any $M>\max\{M_1^2, M_2\}$ and any $x\in\R/\Z$, let $0\le j\le M_1$ be first time $x+j\a\in I_{n-1}(t_0)$. Then we clearly have
$$
\|A_M(x, t_0)\|>\|A_{j}(x, t_0)\|^{-1}\cdot\|A_{M-j}(x+j\a, t_0)\|\ge\l^{-M_1}\l^{\frac13(M-j)}>\l^{\frac14M},
$$
which clearly implies (\ref{ueg}) in our setting. This completes the proof of Theorem B, hence, Main Theorem A.

\end{proof}

\appendix

\section{Generalization of Theorem B}\label{s.generalization}

As discussed in \cite{Wang-Zhang}, Theorem~\ref{t.iteration} can be applied to any one parameter family of cocycle maps  $B\in C^2(\R/\Z\times T, \mathrm{SL}(2,\R))$ such that we could get started with the induction. Here $T\subset\R$ is any compact interval of parameters. In particular, consider
$$
B^{(t,\l)}=\L(x)\circ R_{\psi(x,t)}=\begin{pmatrix}\l(x)&0\\0&\l^{-1}(x)\end{pmatrix}\cdot\begin{pmatrix}\cos\psi(x,t)&-\sin\psi(x,t)\\
\sin\psi(x,t)&\cos\psi(x,t)\end{pmatrix}
$$
with $\psi(x,t)\in C^2(\R/\Z\times T,\R)$, $\l(x)\in C^2(\R/\Z,\R)$. Assume $\l(x)$ and $\psi(x,t)$ satisfying the following conditions:
\begin{itemize}
\item first, $\l(x)>\l$ and $\left|\frac{d^m\l(x)}{dx^m}\right|<C\l$ for each $x\in\R/\Z$ and $m=0,1,2$. For each $t$, $\psi(\R/\Z,t)\subset[0,\pi)$ in $\R\PP^1$.

\item secondly, For each $t\in T$, we have that the corresponding set
 $$
 C(t):=\{x:\psi(x,t)-\frac\pi2=\min_{y\in\R/\Z}\left[\psi(y,t)-\frac\pi2\right]\in\R\PP^1\}=\{c_{1}(t),c_{2}(t)\}
 $$
 with the possibility that $c_1(t)=c_2(t)$.

\item Finally, there exists a $r>0$ such that if we consider the interval
    $$
    I(t)=I_1(t)\cup I_2(t)\mbox{ with }I_{j}(t)=B(c_{j}(t),r),\ j=1,2.
    $$
    Then we assume the following.
\begin{itemize}
\item If $I_1(t)\cap I_2(t)=\varnothing$, then $\psi(\cdot,t)$ is of type $\I$ on $I_j(t)$, $j=1,2$. Moreover, if $\psi(\cdot,t)$ is of type $\I_-$ on $I_1(t)$ then it is of type $\I_+$ on $I_2(t)$, vice versa.
\item If $I_1(t)\cap I_2(t)\neq\varnothing$, then $\psi(\cdot,t)$ is of type $\II$ on $I(t)$.
\end{itemize}
\end{itemize}

Now to get the density of $\CU\CH$ as in Theorem B, we only need to further assume that for all $(x,t)\in\R/\Z\times T$,
$$
\left|\frac{\partial\psi}{\partial t}(x,t)\right|>c.
$$

Then we have the following Corollary.
\begin{corollary}\label{c.general}
For the given $B^{(t,\l)}$ as above, for each $\a\in DC_\tau$ with $\tau>2$, there exists a $\l_0=\l_0(\a,B)$ such that for all $\l>\l_0$,
$$
\{t:(\a,B^{(t,\l)})\in\CU\CH\}
$$
is open and dense in $T$.
\end{corollary}

Clearly, the density of $\CU\CH$ in Corollary~\ref{c.rotation_family}, and Corollary~\ref{c.local} are direct consequences of Corollary~\ref{c.general}. Now we apply Corollary~\ref{c.rotation_family} to the Szeg\H o cocycles which arise naturally in the study of orthogonal polynomial on the unit circle. For detailed introduction of orthogonal polynomial on the unit circle, see \cite{simon1} and \cite{simon2}.

The cocycle map $A^{(E,f)}:\R/\Z\rightarrow \mathrm{SU}(1,1)$ is given by
\beq\label{szego}
A^{(E,f)}(x)=(1-|f(x)|^2)^{-1/2}\begin{pmatrix}\sqrt E& \frac{-\overline{f(x)}}{\sqrt E}\\ f(x)\sqrt E&\frac{1}{\sqrt E} \end{pmatrix},
\eeq
where $E\in\partial\D$, $\D$ is the open unit disk in complex plane $\C$, and $f:\R/\Z\rightarrow\D$ is a measurable function satisfying
$$
\int_X\ln(1-|f|)d\mu>-\infty.
$$
$\mathrm{SU}(1, 1)$ is the subgroup of $\mathrm{SL}(2,\C)$ preserving the unit disk in $\C\PP^1=\C\cup\{\infty\}$ under M\"obius transformations. It is conjugate in $\mathrm{SL}(2,\C)$ to $\mathrm{SL}(2,\R)$ via
$$
Q=\frac{-1}{1+i}\begin{pmatrix}1& -i\\ 1& i \end{pmatrix}\in \mathbb U(2).
$$
In other words, $Q^*\mathrm{SU}(1,1)Q=\mathrm{SL}(2,\R)$. Now consider a function $\t\in C^2(\R/\Z,\R)$ such that $\t(\R/\Z)\subset [0,\frac12)$ and
for some \emph{Diophantine} $\a$, $\t(x)-\t(x-\a)$ is of the same type of function with $v$ in Theorem~A.

One easy example is that $\t(x)=\frac12\cos(x)$, of which $\t(x)-\t(x-\a)$ is of the same type of function with $v$ for all irrational $\a$. Then, we have the follow corollary of Corollary~\ref{c.rotation_family}.
\begin{corollary}\label{c.szego}
Let $f=\l e^{2\pi i[\t(x)+kx]}$, $0<\l<1$, $k\in\Z$ with $\t$ satisfying the above conditions. Let $\a$ be a Diophantine number such that $\t(x)-\t(x-\a)$ is of the same type of function with $v$ in Theorem~A. Then for each such Diophantine $\a$, there exists a $\l_0=\l_0(\t, \a)\in (0,1)$ such that for all $\l_0<\l<1$,
$$
\{E:(\a,A^{(E,f)})\in\CU\CH\}
$$
is open and dense in $\partial\D$.
\end{corollary}
\begin{proof}
Transform $\mathrm{SU}(1,1)$ to $\mathrm{SL}(2,\R)$, set $E=e^{2\pi t}$ for $0\le t< 1$ and do the polar decomposition. We see that the cocycle map (\ref{szego}) can be transformed into the following form
$$
A^{(E,f)}=\begin{pmatrix}\sqrt{\frac{1+\l}{1-\l}}&0\\0&\sqrt{\frac{1-\l}{1+\l}}\end{pmatrix}\cdot R_{\psi(x)}\cdot R_{\pi t},
$$
where $\psi(x,t)=\pi[\theta(x)-\theta(x-\alpha)+k\alpha]$. This concludes the proof.
\end{proof}

\end{document}